\documentclass[12pt]{amsart}
\usepackage{amssymb,latexsym}
\usepackage{enumerate}
\usepackage[cp1250]{inputenc}

\makeatletter
\@namedef{subjclassname@2010}{%
  \textup{2010} Mathematics Subject Classification}
\makeatother

\newtheorem{thm}{Theorem}[section]
\newtheorem{cor}[thm]{Corollary}
\newtheorem{lem}[thm]{Lemma}

\newtheorem{prop}[thm]{Proposition}

\theoremstyle{definition}
\newtheorem{defin}[thm]{Definition}
\newtheorem{rem}[thm]{Remark}
\newtheorem{exa}[thm]{Example}

\numberwithin{equation}{section}

\begin{document}

\baselineskip=17pt

\title[Ranks of $\mathcal{F}$-limits of filter sequences]{Ranks of $\mathcal{F}$-limits of filter sequences}

\author[A. Kwela]{Adam Kwela}
\address{Institute of Mathematics\\ Polish Academy of Sciences\\ul. Śniadeckich 8, 00-956 Warszawa,
Poland}
\email{A.Kwela@impan.pl}

\author[I. Recław]{Ireneusz Recław}
\address{Institute of Informatics\\ University of Gdańsk\\ul. Wita Stwosza 57, 80-952 Gdańsk,
Poland}

\date{}

\begin{abstract}
We give an exact value of the rank of an $\mathcal{F}$-Fubini sum of filters for the case where $\mathcal{F}$ is a Borel filter of rank $1$. We also consider $\mathcal{F}$-limits of filters $\mathcal{F}_i$, which are of the form $$\lim_\mathcal{F}\mathcal{F}_i=\left\{A\subset X: \left\{i\in I: A\in\mathcal{F}_i\right\}\in\mathcal{F}\right\}.$$ We estimate the ranks of such filters; in particular we prove that they can fall to $1$ for $\mathcal{F}$ as well as for $\mathcal{F}_i$ of arbitrarily large ranks. At the end we prove some facts concerning filters of countable type and their ranks.
\end{abstract}


\maketitle

\section{Introduction}

Let $I$ be a countable set. A family of sets $\mathcal{F}\subset\mathcal{P}(I)$ is a filter on $I$ if it is closed under taking finite intersections and supersets. Throughout this paper we denote 
$$\mathcal{A}^*=\left\{A\subset I: I\setminus A\in\mathcal{A}\right\},$$
for a family $\mathcal{A}\subset\mathcal{P}(I)$. If $\mathcal{F}$ is a filter on $I$, then $\mathcal{F}^*$ is an ideal (i.e., a family closed under taking finite unions and subsets) called the dual ideal of the filter $\mathcal{F}$. We denote $Fin=\left[\omega\right]^{<\omega}$ and $\mathcal{F}_{Fr}=Fin^*$. Clearly $Fin$ is an ideal and $\mathcal{F}_{Fr}$ is its dual filter, called the Fr\'echet filter. We say that $\mathcal{B}$ is a basis which generates filter $\mathcal{F}$ if $\mathcal{F}=\left\{M:\exists_{B\in\mathcal{B}}B\subset M\right\}$. In the sequel we treat filters and ideals on $X$ as subsets of the Cantor space (by canonical identification, $\mathcal{P}(X)\approx 2^X$), so we can speak about the descriptive complexity of filters and ideals.

Consider two subsets $A$ and $B$ of a Polish space $X$ and a family of sets $\Gamma\subset\mathcal{P}(X)$. We say that $A$ is $\Gamma$-separated from $B$ if there exists a subset $S\subset X$ in $\Gamma$ for which $A\subset S$ and $B\cap S=\emptyset$.

The rank of an analytic filter $\mathcal{F}$ is the ordinal
$${\rm rk}(\mathcal{F})=\min\left\{\alpha<\omega_{1}: \mathcal{F}\mbox{\boldmath{ is }}\Sigma^{0}_{1+\alpha} \mbox{\boldmath{-separated from }} \mathcal{F}^{*}\right\}.$$
By the Lusin Separation Theorem (cf. \cite{Kechris}), every analytic filter has countable rank. The concept of filter rank was introduced by Debs and Saint Raymond \cite{Debs}, although it was also studied by Solecki \cite{Solecki} in the context of a question raised by Dobrowolski and Marciszewski \cite{Marciszewski}. Clearly, for every filter $\mathcal{F}$ its rank ${\rm rk}(\mathcal{F})$ is unique. 

Let $(X,\rho)$ be a metric space and $\mathcal{F}$ be a filter on a countable set $I$. A sequence $(x_i)_{i\in I}\in X^{I}$ is convergent to $x\in X$ relatively to $\mathcal{F}$ ($x=\mathcal{F}-\lim x_i$) if for every $\epsilon>0$ we have 
$$\left\{i\in I:\rho(x_i,x)<\epsilon\right\}\in\mathcal{F}.$$
We write $f=\mathcal{F}-\lim f_i$ and say that function $f:X\rightarrow\mathbb{R}$ is a limit of the sequence of functions $(f_i:X\rightarrow\mathbb{R})_{i\in I}$ relatively to $\mathcal{F}$ if $f(x)=\mathcal{F}-\lim f_i(x)$ for each $x\in X$. By $\mathcal{C}_{\mathcal{F}}(X)$ we denote the family of all real-valued functions on the space $X$, which can be represented as a limit relatively to $\mathcal{F}$ of a sequence of continuous functions. By $\mathcal{B}_{\alpha}(X)$ we denote the family of all real-valued functions on the space $X$ of Borel class $\alpha<\omega_1$, i.e., functions $f:X\rightarrow\mathbb{R}$ such that $f^{-1}[U]\in\Sigma^0_{1+\alpha}$ for any open subset $U$ of $\mathbb{R}$.

Filter rank appears to play a fundamental role in studying the class of pointwise limits relatively to a filter of sequences of continuous functions. This is apparent from the following theorem of Debs and Saint Raymond \cite{Debs}:

\begin{thm}
\label{ThmRank}
Let $\mathcal{F}$ be an analytic filter and $\alpha<\omega_{1}$ be a countable ordinal. Then\\
(a) $\mathcal{C}_{\mathcal{F}}\left(X\right)\subset\mathcal{B}_{\alpha}\left(X\right)$ for any Polish space $X$ if and only if ${\rm rk}(\mathcal{F})\leq\alpha$.\\
(b) $\mathcal{C}_{\mathcal{F}}\left(X\right)\supset\mathcal{B}_{\alpha}\left(X\right)$ for any zero-dimensional Polish space $X$ if and only if ${\rm rk}(\mathcal{F})\geq\alpha$.\\
(c) $\mathcal{C}_{\mathcal{F}}\left(X\right)=\mathcal{B}_{\alpha}\left(X\right)$ for any zero-dimensional Polish space $X$ if and only if ${\rm rk}(\mathcal{F})=\alpha$.
\end{thm}

If $\left(X_i\right)_{i\in I}$ is a family of sets, then we denote by $\Sigma_{i\in I}X_i$ its disjoint sum, i.e., the set of all pairs $(i,x)$ where $i\in I$ and $x\in X_i$.

Recall that for a filter $\mathcal{G}$ on $I$ and a family of filters $\left(\mathcal{F}_i\right)_{i\in I}$, a collection of all sets of the form
$$\sum_{i\in G}F_i,$$
for $G\in\mathcal{G}$ and $F_i\in\mathcal{F}_i$ constitutes a basis of a filter on the set $\sum_{i\in I}{\rm dom}\left(\mathcal{F}_i\right)$. We denote this filter by $\mathcal{G}-\sum_{i\in I}\mathcal{F}_i$ and call it the $\mathcal{G}$-Fubini sum of the family $\left(\mathcal{F}_i\right)_{i\in I}$. In particular, if all $\mathcal{F}_i$ are the same filter $\mathcal{F}$, then we get the product of filters $\mathcal{G}\times\mathcal{F}$. 

Debs and Saint Raymond \cite[Prop. 4.4]{Debs} studied ranks of $\mathcal{F}$-Fubini sums obtaining the following results:

\begin{thm}
\label{ThmDebsA}
Let $\mathcal{G}$ be the $\mathcal{F}$-Fubini sum of $\left(\mathcal{F}_i\right)_{i\in I}$  and $J\subset I$ be an element of the filter $\mathcal{F}$.\\
(a) If ${\rm rk}(\mathcal{F})\geq\alpha$ and ${\rm rk}(\mathcal{F}_i)\geq\xi$ for all $i\in J$, then ${\rm rk}(\mathcal{G})\geq\xi+\alpha$.\\
(b) If ${\rm rk}(\mathcal{F})\leq\alpha$ and ${\rm rk}(\mathcal{F}_i)\leq\xi$ for all $i\in J$, then ${\rm rk}(\mathcal{G})\leq\xi+1+\alpha$.\\
(c) If $\mathcal{F}=\mathcal{F}_{Fr}$ and ${\rm rk}(\mathcal{F}_i)\leq\xi$ for all $i\in J$, then ${\rm rk}(\mathcal{G})\leq\xi+1$.
\end{thm}

Let $\mathcal{F}$ be a filter on a set $I$ and $\left(\mathcal{F}_i\right)_{i\in I}$ be a family of filters on $X$. Then
$$\lim_{i\rightarrow\mathcal{F}}\mathcal{F}_i=\left\{A\subset X: \left\{i\in I: A\in\mathcal{F}_i\right\}\in\mathcal{F}\right\}$$
is a filter on $X$. We call it the $\mathcal{F}$-limit of filters $\left(\mathcal{F}_i\right)_{i\in I}$ and sometimes denote it also by $\lim_{\mathcal{F}}\mathcal{F}_i$. The $\mathcal{F}$-limits were studied by Fremlin \cite{Fremlin} in the context of filters of countable type.

Our goal is to prove estimates similar to those in Theorem \ref{ThmDebsA} for $\mathcal{F}$-limits. In particular, we give an upper estimate of the ranks of such limits (Theorem \ref{ThmA}). This enables us to improve part (b) of Theorem \ref{ThmDebsA} when $\mathcal{F}$ is a Borel filter of rank $1$ (Corollary \ref{CorFubini}). 

The lower estimate of the ranks of $\mathcal{F}$-limits cannot be obtained in the way presented above. Namely, we prove that the rank of an
$\mathcal{F}$-limit can be equal to $1$, even for $\mathcal{F}$ and $\mathcal{F}_i$ of arbitrarily large ranks (Theorem \ref{ThmB}).

\section{Preliminaries}

\subsection{Basic properties of filters}

A filter $\mathcal{F}$ is free if $\bigcap\mathcal{F}=\emptyset$. It is principal if it is of the form 
$$\mathcal{F}_E=\left\{A\subset I: E\subset A\right\},$$ for a certain subset $E\subset I$. Maximal filters are called ultrafilters. They can be characterized by the following condition: A filter $\mathcal{F}$ is an ultrafilter if and only if for each $A$ either $A\in\mathcal{F}$ or $X\setminus A\in\mathcal{F}$. All principal ultrafilters are of the form
$$\mathcal{F}_{\left\{x\right\}}=\left\{A\subset I: x\in A\right\},$$
for a certain $x\in I$. If $A\setminus B$ is finite, then we write $A\subset^* B$ and say that $A$ is almost contained in $B$. Recall that filter $\mathcal{F}$ is a P-filter if for every family $\left\{A_n:n\in\omega\right\}$ of sets in $\mathcal{F}$ there is a single set $A\in\mathcal{F}$ such that $A\subset^* A_n$ for all $n\in\omega$. Two filters $\mathcal{F}$ and $\mathcal{G}$ are isomorphic provided that there is a bijection $\sigma:{\rm dom}(\mathcal{G})\rightarrow {\rm dom}(\mathcal{F})$ such that for all $A\subset{\rm dom}(\mathcal{F})$
$$A\in\mathcal{F}\Leftrightarrow \sigma^{-1}[A]\in\mathcal{G}.$$
If $\mathcal{F}$ and $\mathcal{G}$ are filters, then a quasi-homomorphism from $\mathcal{F}$ to $\mathcal{G}$ is a mapping $\pi:F\rightarrow {\rm dom}\left(\mathcal{G}\right)$ where $F\in\mathcal{F}$, such that for each $G\in\mathcal{G}$ its preimage under $\pi$ belongs to $\mathcal{F}$ (i.e., $\pi^{-1}[G]\in\mathcal{F}$). 

\subsection{Properties of filter rank}

\begin{prop}
Filter rank has the following properties:
\begin{itemize}
	\item ${\rm rk}(\mathcal{F})=\min\left\{\alpha<\omega_{1}: \mathcal{F}\mbox{\boldmath{ is }}\Pi^{0}_{1+\alpha} \mbox{\boldmath{-separated from }} \mathcal{F}^{*}\right\}$.
	\item If $\mathcal{F}$ and $\mathcal{G}$ are filters and $\mathcal{F}\subset\mathcal{G}$, then ${\rm rk}(\mathcal{F})\leq{\rm rk}(\mathcal{G})$. 
	\item A filter is of rank $0$ if and only if it is not free.
	\item If there exists a quasi-homomorphism from $\mathcal{F}$ to $\mathcal{G}$, then ${\rm rk}(\mathcal{G})\leq{\rm rk}(\mathcal{F})$.
	\item Two isomorphic filters have the same rank.
\end{itemize}
\end{prop}

The first property in the above Proposition is a simple consequence of the fact that the canonical involution $A\mapsto A^c$ is a homeomorphism. Proofs of the remaining properties can be found in \cite{Debs}.

\subsection{Katětov filters}
\label{SubsectionKatetov}

For $\alpha<\omega$, Katětov filters $\mathcal{N}_{\alpha}$ are defined inductively by:
\begin{itemize}
	\item $\mathcal{N}_{0}=\left\{\left\{0\right\}\right\}$ is a unique filter on the set $\left\{0\right\}$,
	\item $\mathcal{N}_{\alpha+1}=\mathcal{F}_{Fr}\times\mathcal{N}_{\alpha}$,
\end{itemize}

The idea of such filters comes from Katětov (\cite{Kat}) and Grimeisen (\cite{Grimeisen}), who worked on this problem independently. Note that actually $\mathcal{N}_1$ is the Fr\'echet filter $\mathcal{F}_{Fr}$ and $\mathcal{N}_2$ is the dual filter of the well known and extensively studied ideal $Fin\times Fin$. These filters play fundamental role in the theory of ranks of filters, which follows from:

\begin{thm}\cite[Thm. 6.5]{Debs}\\
\label{ThmDebsB}
For every $\alpha<\omega$, the filter $\mathcal{N}_\alpha$ has rank $\alpha$ (and therefore generates the Borel class $\mathcal{B}_\alpha\left(X\right)$, for any zero-dimensional Polish space $X$ i.e., $\mathcal{C}_{\mathcal{N}_\alpha}\left(X\right)=\mathcal{B}_\alpha\left(X\right)$).
\end{thm}

\subsection{Isomorphic copies of filters}

Let $\mathcal{F}$ and $\mathcal{G}$ be filters. We shall write $\mathcal{F}\sqsubseteq\mathcal{G}$ and say that $\mathcal{G}$ contains an isomorphic copy of $\mathcal{F}$ if there exists a bijection $\sigma:{\rm dom}(\mathcal{F})\rightarrow{\rm dom}(\mathcal{G)}$ such that $\sigma(A)\in\mathcal{G}$ for every element $A\in \mathcal{F}$.

For two given filters $\mathcal{F}$ and $\mathcal{G}$, if $\mathcal{F}\sqsubseteq\mathcal{G}$, then ${\rm rk}(\mathcal{F})\leq{\rm rk}(\mathcal{G})$ \cite[Lem. 7.2]{Debs}. 

\begin{thm}
\label{ThmIsomCopies}
For $\alpha\in\left\{1,2\right\}$, an analytic filter $\mathcal{F}$ is of rank greater than or equal to $\alpha$ if and only if it contains an isomorphic copy of $\mathcal{N}_{\alpha}$.
\end{thm}

The proof can be found in \cite[Prop. 7.3 and Thm. 7.5]{Debs}. Also in \cite[Thm. 4]{Reclaw} the authors prove a similar result, but with an additional assumption: that $\mathcal{F}$ is Borel.

\subsection{Filters of countable type}

The idea of filters of countable type was introduced by Mauldin, Preiss and Weizs\"acker in \cite{Preiss}. The family of filters of countable type on $X$ is the smallest class of filters on $X$ containing the principal ultrafilters and closed under the operations of countable intersection and increasing countable union. The theory of filters of countable type was studied also by Fremlin, who gave the following equivalent definition \cite{Fremlin}:

\begin{defin}
We define families $\mathfrak{F}_\alpha$ inductively:
\begin{itemize}
	\item $\mathfrak{F}_0$ is the family of all principal ultrafilters on $X$.
	\item $\mathfrak{F}_\alpha$ is the family of all filters on $X$, which are equal to $\lim_{i\rightarrow\mathcal{F}_{Fr}}\mathcal{F}_i$, where $\mathcal{F}_i\in\bigcup_{\xi<\alpha}\mathfrak{F}_\xi$ for every $i\in\omega$.
\end{itemize}
Then $\mathfrak{F}_C=\bigcup_{\alpha<\omega_1}\mathfrak{F}_\alpha$ is the set of all filters of countable type.\\
Moreover, for a filter $\mathcal{F}$ of countable type ${\rm ct}(\mathcal{F})=\min\left\{\alpha<\omega_1:\mathcal{F}\in\mathfrak{F}_\alpha \right\}$.
\end{defin}

\section{The upper estimate of the ranks of limits of filter sequences}

In this Section we investigate the ranks of $\mathcal{F}$-limits of filters. We can estimate from above the Borel class and ranks of limits of filters.

\begin{prop}
\label{Prop1}
Let $\mathcal{F}$ be a filter on $I$ and $\left(\mathcal{F}_i\right)_{i\in I}$ be a family of filters on $X$. Let also $\Gamma^0_\xi$ denote $\Sigma^0_\xi$ or $\Pi^0_\xi$. 
		\item a) Suppose that there is $J\in\mathcal{F}$ such that for each $i\in J$ we have $\mathcal{F}_i\in\Gamma^0_{\beta_i}$ for a certain $\beta_i<\beta$. If $\mathcal{F}\in\Gamma^0_{1+\alpha}$, then $\lim_{\mathcal{F}}\mathcal{F}_i\in\Gamma^0_{1+\beta+\alpha}$.
		\item b) Suppose that there is $J\in\mathcal{F}$ such that for each $i\in J$ we have ${\rm rk}(\mathcal{F}_i)\leq\beta$. If ${\rm rk}(\mathcal{F})\leq\alpha$, then ${\rm rk}\left(\lim_{\mathcal{F}}\mathcal{F}_i\right)\leq\beta+1+\alpha$.
\end{prop}

\begin{proof}
In order to prove part a) we define functions $\varphi_i:\mathcal{P}(X)\rightarrow 2$ by
$$\varphi_i(A)=\left\{\begin{array}{cl}
1 & $, if  $i\in J$ and $A\in\mathcal{F}_i\\
0 & $, if  not $\\
\end{array}\right.$$
Every such a function is of class $\beta$. The function $\varphi:\mathcal{P}(X)\rightarrow\mathcal{P}(I)$ defined as $\varphi(A)=\left\{i\in I: \varphi_i(A)=1\right\}$ is also of class $\beta$. Since $\mathcal{F}\in\Gamma^0_{1+\alpha}(\mathcal{P}(I))$ we have $\lim_{\mathcal{F}}\mathcal{F}_i=\varphi^{-1}[\mathcal{F}]\in\Gamma^0_{1+\beta+\alpha}(\mathcal{P}(X))$.\\
The proof of part b) is similar: let $S\in\Sigma^0_{1+\alpha}$ be a set separating the filter $\mathcal{F}$ from its dual ideal and $S_i\in\Sigma^0_{1+\beta}$ be sets separating the filters $\mathcal{F}_i$ from their dual ideals. Define functions $\psi_i:\mathcal{P}(X)\rightarrow 2$ by
$$\psi_i(A)=\left\{\begin{array}{cl}
1 & $, if  $i\in J$ and $A\in S_i\\
0 & $, if not $\\
\end{array}\right.$$
These functions are of class $\beta+1$. The function $\psi:\mathcal{P}(X)\rightarrow\mathcal{P}(I)$ defined as $\psi(A)=\left\{i\in I: \psi_i(A)=1\right\}$ is also of class $\beta+1$. Then the set $\psi^{-1}[S]$  belongs to $\Sigma^0_{1+\beta+1+\alpha}(\mathcal{P}(X))$. Now it is sufficient to show that $\psi^{-1}[S]$ separates $\lim_{\mathcal{F}}\mathcal{F}_i$ from its dual ideal.\\
Because $\mathcal{F}_i\subset S_i$, we have that $\varphi_i[A]=1$ implies $\psi_i[A]=1$ for every $i$ and $A$.
It follows in turn that $\lim_{\mathcal{F}}\mathcal{F}_i=\varphi^{-1}[\mathcal{F}]\subset\psi^{-1}[\mathcal{F}]$. Since $\mathcal{F}\subset S$, we have $\lim_{\mathcal{F}}\mathcal{F}_i\subset\psi^{-1}[S]$.\\
On the other hand, if $A\in\left(\lim_{\mathcal{F}}\mathcal{F}_i\right)^*$, then $\left\{i\in I: A\in\mathcal{F}^*_i\right\}\in\mathcal{F}$ and $\left\{i\in I: A\notin\mathcal{F}^*_i\right\}\in\mathcal{F}^*$. Notice that $$\psi(A)\subset\left\{i\in I: A\in S_i\right\}\subset\left\{i\in I: A\notin\mathcal{F}^*_i\right\}.$$
Hence $\psi(A)\in\mathcal{F}^*$ and $\psi(A)\notin S$.
\end{proof}

Actually in several cases we can estimate the ranks of limits of filters even more precisely. But before showing this, we introduce a game $G(\mathcal{F})$.

For a filter $\mathcal{F}$, a set $Z=\left\{Z^k:k\in\omega\right\}\subset Fin\setminus\left\{\emptyset\right\}$ is called $\mathcal{F}$-universal if for every element $M\in\mathcal{F}$ there is $k\in\omega$ such that $Z^k\subset M$. We say that $\mathcal{F}$ is $\omega$-diagonalizable by $\mathcal{F}$-universal sets if there are $\mathcal{F}$-universal sets $Z_n=\left\{Z^k_n:k\in\omega\right\}$ such that
$$\forall_{M\in\mathcal{F}}\exists_{n\in\omega}\forall_{k\in\omega}^{\infty}Z_n^k\cap M\neq\emptyset.$$

Next, for a filter $\mathcal{F}$, the game $G(\mathcal{F})$ is defined as follows: in the $n$th move Player I plays an element $C_n\in\mathcal{F}$ and then Player II plays a finite subset of integers $F_n\subset C_n$. Player I wins when $\bigcup_{n\in\omega}F_n\in\mathcal{F}^*$; otherwise Player II wins. This game was investigated by Laflamme \cite[Thm. 2.16]{Laflamme} and Recław and Laczkovich \cite{Reclaw}, who gave the following characterizations:

\begin{lem}
\label{Lem1}
Player II has a winning strategy in $G(\mathcal{F})$ if and only if $\mathcal{F}$ is $\omega$-diagonalizable by $\mathcal{F}$-universal sets.
\end{lem}

\begin{lem}
\label{Lem2}
Player I has a winning strategy in $G(\mathcal{F})$ if and only if $\mathcal{F}$ contains an isomorphic copy of $\mathcal{N}_2$.
\end{lem}

\begin{lem}
\label{Lem3}
If $\mathcal{F}$ is a Borel filter, then $G(\mathcal{F})$ is a determined game, i.e., one of the players has a winning strategy.
\end{lem}

Now we can proceed to give a better upper estimate of the ranks of $\mathcal{F}$-limits.

\begin{thm}
\label{ThmA}
Let $\mathcal{F}$ be a Borel filter of rank $1$ on a set $I$, $J\in\mathcal{F}$ and $\left(\mathcal{F}_i\right)_{i\in I}$ be a family of filters on a set $X$ such that for all $i\in J$ the rank of the filter $\mathcal{F}_i$ is less than or equal to $\alpha<\omega_1$. Then ${\rm rk}\left(\lim_{\mathcal{F}}\mathcal{F}_i\right)\leq\alpha+1$.
\end{thm}

\begin{proof}
Let $S_i\in\Pi^0_{1+\alpha}(X)$ be sets separating the filters $\mathcal{F}_i$ from their dual ideals.
Since $\mathcal{F}$ is a Borel filter of rank $1$, by Lemma \ref{Lem3} the game $G(\mathcal{F})$ is determined, and by Lemma \ref{Lem2} together with Theorem \ref{ThmIsomCopies} Player II has a winning strategy. Then, by Lemma \ref{Lem1} there is a family $\left\{Z_n^k:k\in\omega\right\}_{n\in\omega}$ of $\mathcal{F}$-universal sets $\omega$-diagonalizing $\mathcal{F}$.
Let
$$S=\bigcup_{n\in\omega}\bigcup_{m\in\omega}\bigcap_{k>m}\bigcup_{i\in Z_n^k\cap J}S_i.$$
This set is of additive class $1+\alpha+1$.\\
Moreover, if $A\in\lim_{\mathcal{F}}\mathcal{F}_i$, then the set $\left\{i\in I:A\in S_i\right\}\cap J$ is an element of $\mathcal{F}$. So, there is $n\in\omega$ such that $Z_n^k\cap\left\{i\in I:A\in S_i\right\}\cap J\neq\emptyset$ for all but finitely many $k\in\omega$. Hence
$$\exists_{n\in\omega}\exists_{m\in\omega}\forall_{k>m}\exists_{i\in Z_n^k\cap J}A\in S_i$$
and $A\in S$.\\
Conversely, if $A\in\left(\lim_{\mathcal{F}}\mathcal{F}_i\right)^*$, then the set $\left\{i\in I:A\notin S_i\right\}$ is an element of $\mathcal{F}$. So, for every $n\in\omega$ there is $k\in\omega$ such that $Z_n^k\cap J\subset Z_n^k\subset\left\{i\in I:A\notin S_i\right\}$. Since $\mathcal{F}$ has rank $1$ (contains the Fr\'echet filter) and the sets $Z_n^k$ are finite, we have an even stronger claim: for every $n\in\omega$ there are infinitely many $k\in\omega$ such that $Z_n^k\cap J\subset\left\{i\in I:A\notin S_i\right\}$. Hence
$$\forall_{n\in\omega}\forall_{m\in\omega}\exists_{k>m}\forall_{i\in Z_n^k\cap J}A\notin S_i$$
and $A\notin S$.
\end{proof}

Besides giving an exact outcome for the ranks of some Fubini sums, the next Corollary shows that Theorem \ref{ThmA} is the best upper estimate for the ranks of $\mathcal{F}$-limits for $\mathcal{F}$ a Borel filter of rank $1$ that we can get.

\begin{cor}
\label{CorFubini}
Let $\mathcal{F}$ be a Borel filter of rank $1$ on a set $I$ and $\mathcal{F}_i$ for $i\in I$ be filters on $D_i$, respectively. Let also $\mathcal{G}=\mathcal{F}-\sum_{i\in I}\mathcal{F}_i$ and $J\in\mathcal{F}$. If ${\rm rk}(\mathcal{F}_i)=\alpha$ for all $i\in J$, then ${\rm rk}(\mathcal{G})=\alpha+1$.
\end{cor}

\begin{proof}
According to Prop. 4.4 from \cite{Debs} ${\rm rk}(\mathcal{G})\geq\alpha+1$. Notice that the $\mathcal{F}$-Fubini sum of the family $\left(\mathcal{F}_i\right)_{i\in I}$ can be represented as an $\mathcal{F}$-limit of the family of filters $\left(\tilde{\mathcal{F}}_i\right)_{i\in I}$, defined by
$$\tilde{\mathcal{F}}_j=\left\{M\subset\sum_{i\in I}D_i : \left\{x\in D_j:(j,x)\in M\right\}\in\mathcal{F}_j\right\}.$$
It is easy to see that ${\rm rk}(\tilde{\mathcal{F}}_i)={\rm rk}(\mathcal{F}_i)$,for each $i\in I$. Then Theorem \ref{ThmA} implies that ${\rm rk}(\mathcal{G})\leq\alpha+1$, which finishes the proof.
\end{proof}

Immediately from induction based on Theorem \ref{ThmA}, we obtain also the following Corollary concerning filters of countable type:

\begin{cor}
If $n\in\omega$ and $\mathcal{F}\in\mathfrak{F}_n$, then ${\rm rk}(\mathcal{F})\leq n$.
\end{cor}

\section{The lower estimate of ranks of limits of filter sequences}

\begin{prop}
\label{PropBottomEst}
Let $\mathcal{F}$ be a filter on $I$ and $\left(\mathcal{F}_i\right)_{i\in I}$ be a family of filters on $X$. If there is $J\in\mathcal{F}$ such that ${\rm rk}\left(\mathcal{F}_i\right)\geq 1$ for each $i\in J$, then ${\rm rk}\left(\lim_\mathcal{F} \mathcal{F}_i\right)\geq 1$.
\end{prop}

\begin{proof}
Since ${\rm rk}\left(\mathcal{F}_i\right)\geq 1$ for $i\in J$, these filters are free and therefore contain the Fr\'echet filter. Since $J\in\mathcal{F}$ the filter $\lim_\mathcal{F} \mathcal{F}_i$ also contains the Fr\'echet filter, and so has positive rank.
\end{proof}

\begin{rem}
The family of Katětov filters defined in Subsection \ref{SubsectionKatetov} can be extended to $\omega_1$. We omit the definition of filters $\mathcal{N}_\gamma$ for $\omega\leq\gamma<\omega_1$; however, it can be found in \cite{Debs}. Theorem \ref{ThmDebsB} can also be generalized: for every $\alpha<\omega_1$, the filter $\mathcal{N}_\alpha$ has rank $\alpha$ (cf. \cite[Thm. 6.5]{Debs}). 
\end{rem}

Now we can proceed to show that the estimate from Proposition \ref{PropBottomEst} is the best of this kind that we are able to get.

\begin{lem}
\label{LemRank2}
Filters $\mathcal{N}_\gamma$ for $1\leq\gamma<\omega_1$ have the following property: there is an infinite family of pairwise disjoint infinite sets $\left(Z_i\right)_{i\in\omega}$ such that for any subset $M\subset{\rm dom}\left(\mathcal{N}_\gamma\right)$, if $Z_i\setminus M$ is infinite for all $i\in\omega$, then $M\notin\mathcal{N}_\gamma$.
\end{lem}

The above Lemma is a simple conclusion based on the definition of filters $\mathcal{N}_\gamma$; therefore we omit the proof.

\begin{lem}
\label{LemRank1}
For every positive ordinal number $\alpha<\omega_1$, there exist two filters $\mathcal{G}_0$ and $\mathcal{G}_1$ of rank $\alpha$ such that $\mathcal{G}_0\cap\mathcal{G}_1$ has rank $1$.
\end{lem}

\begin{proof}
Set $\alpha<\omega_1$ and let $\left(Z_i\right)_{i\in\omega}$ be the family from Lemma \ref{LemRank2} for the filter $\mathcal{N}_\alpha$. Let now $\pi_0,\pi_1:\omega\rightarrow{\rm dom}\left(\mathcal{N}_\alpha\right)$ be any bijections satisfying $\left|\pi_0^{-1}[Z_i]\cap\pi_1^{-1}[Z_j]\right|=\omega$, for all $i,j\in\omega$. Define filters $\mathcal{G}_0$ and $\mathcal{G}_1$ as follows:
$$M\in\mathcal{G}_k\Leftrightarrow\pi_k [M]\in\mathcal{N}_\alpha,$$
for $k\in\left\{0,1\right\}$. They are isomorphic with $\mathcal{N}_\alpha$ and therefore have rank $\alpha$. Notice that $\mathcal{G}_0\cap\mathcal{G}_1$ contains $\mathcal{N}_1$ (since filters $\mathcal{G}_0$ and $\mathcal{G}_1$ both contain it), so ${\rm rk}(\mathcal{G}_0\cap\mathcal{G}_1)>0$. \\
Observe also that for $k\in\left\{0,1\right\}$ the family $\left(\pi_k^{-1}[Z_i]\right)_{i\in\omega}$ for $\mathcal{G}_k$ has the same property as the family $\left(Z_i\right)_{i\in\omega}$ for the filter $\mathcal{N}_\alpha$.\\
The filter $\mathcal{G}_0\cap\mathcal{G}_1$ is analytic, so by using Theorem \ref{ThmIsomCopies} it suffices to show that $\mathcal{N}_2\not\sqsubseteq\mathcal{G}_0\cap\mathcal{G}_1$. Let $\tau:\omega\rightarrow\omega^2$ be any bijection and denote $E_i=\tau^{-1}\left[\left\{i\right\}\times\omega\right]$. There are two possible cases:\\
1. Suppose that there are $k\in\left\{0,1\right\}$ and $i_0\in\omega$ such that $\pi_k^{-1}[Z_{i_0}]$ is covered by finitely many $E_i$, i.e., $\pi_k^{-1}[Z_{i_0}]\subset\bigcup_{i\in T}E_i$ for some finite set $T$. We can assume that $k=1$. The set $\bigcup_{i\notin T}\left(\left\{i\right\}\times\omega\right)$ is an element of $\mathcal{N}_2$, but its preimage under $\tau$, $\bigcup_{i\notin T}E_i$, is disjoint with $\pi_1^{-1}[Z_{i_0}]$. Therefore for every $j$ we have
$$\pi_0^{-1}[Z_j]\setminus\bigcup_{i\notin T}E_i\supset\pi_0^{-1}[Z_j]\cap\pi_1^{-1}[Z_{i_0}].$$
By the definition of $\pi_0$ and $\pi_1$ the intersection $\pi_1^{-1}[Z_{i_0}]\cap \pi_0^{-1}[Z_{j}]$ is infinite. Hence $\bigcup_{i\notin T}E_i$ is not an element of $\mathcal{G}_0$.\\
2. Suppose that none of the sets $\pi_k^{-1}[Z_{j}]$ are covered by finitely many $E_i$. For each $i\in\omega$ let $S_i$ be a selector of the family
$$\left\{\pi_1^{-1}[Z_i]\cap E_j: j>i \:{\rm and}\: \pi_1^{-1}[Z_i]\cap E_j\neq\emptyset\right\}.$$
Since for each $j\in\omega$ $$\left|E_j\cap\bigcup_{i\in\omega}S_i\right|\leq j<\omega,$$
we have
$$\tau\left[\omega\setminus\bigcup_{i\in\omega}S_i\right]\in\mathcal{N}_2.$$
However, the preimage of this set under $\tau$ is not an element of $\mathcal{G}_1$, since for each $j\in\omega$
$$\left|\bigcup_{i\in\omega}S_i\cap \pi_{1}^{-1}[Z_{j}]\right|=\left|S_{j}\cap \pi_{1}^{-1}[Z_{j}]\right|=\omega.$$
Hence $\mathcal{N}_2\not\sqsubseteq\mathcal{G}_0\cap\mathcal{G}_1$.
\end{proof}

\begin{thm}
\label{ThmB}
For all ordinals $\alpha, \beta<\omega_1$ where $\alpha>0$, there are a filter $\mathcal{F}$ of rank $\beta$ and a family of filters $(\mathcal{F}_i)_{i\in{\rm dom}\left(\mathcal{F}\right)}$ of ranks $\alpha$ such that $\lim_\mathcal{F}\mathcal{F}_i$ has rank $1$.
\end{thm}

\begin{proof}
Let $\mathcal{G}_0$ and $\mathcal{G}_1$ be filters from Lemma \ref{LemRank1} and $\mathcal{F}$ be any (non-maximal) filter of rank $\beta$ (for instance $\mathcal{F}$ can be the filter $\mathcal{N}_\beta$ defined in \cite{Debs}). Since $\mathcal{F}$ is not an ultrafilter, we can find a set $H$ such that $H\notin\mathcal{F}$ and ${\rm dom}\left(\mathcal{F}\right)\setminus H\notin\mathcal{F}$. Set $\mathcal{F}_i=\mathcal{G}_0$ for $i\in H$ and $\mathcal{F}_i=\mathcal{G}_1$ for $i\notin H$. Then $\lim_\mathcal{F}\mathcal{F}_i$ is equal to $\mathcal{G}_0\cap\mathcal{G}_1$, and so is of rank $1$.
\end{proof}

\section{Some facts concerning filters of countable type}

Following Laflamme (cf. \cite{Laflamme}), we call a filter diagonalizable if there is infinite set $A$ such that $A\subset^* M$ for each set $M$ from the filter. A filter is $\omega$-diagonalizable if there is a countable family of infinite sets $\mathcal{A}=\left\{A_n:n\in\omega\right\}$ such that for each set $M$ from the filter, there is $n\in\omega$ with $A_n\subset M$. 

\begin{prop}
If $\mathcal{F}$ is a filter of countable type, then $\mathcal{F}$ either is $\omega$-diagonalizable or
is of the form $\mathcal{F}_A$ for some finite set $A$.
\end{prop}

\begin{proof}
Let $\mathfrak{F}$ be a family of all filters of countable type which either are $\omega$-diagonalizable or of the form $\mathcal{F}_A$ for some finite set $A$. Of course $\mathfrak{F}$ contains all principal ultrafilters since they are of the form $\mathcal{F}_{\left\{k\right\}}$ for some $k\in\omega$. We will show that $\mathfrak{F}$ is closed under countable intersections and increasing countable unions.\\
Let $\mathcal{F}_1,\mathcal{F}_2,\ldots\in\mathfrak{F}$. Assume first that one of them is $\omega$-diagonalizable by some countable family. Then, the filter $\bigcap_{n\in\omega}\mathcal{F}_n$ is also $\omega$-diagonalizable by the same family. Assume now that all of them are of the form $\mathcal{F}_n=\mathcal{F}_{A_n}$ for some finite sets $A_n$. Then $$\bigcap_{n\in\omega}\mathcal{F}_n=\mathcal{F}_{E},$$ where $E=\bigcup_{n\in\omega} A_n$.
If $E$ is finite, then this intersection belongs to $\mathfrak{F}$ by definition. Otherwise the filter  $\bigcap_{n\in\omega}\mathcal{F}_n$ is diagonalized by the set $E$. Hence it is $\omega$-diagonalized by the family $\left\{E\right\}$.\\
Now let $\mathcal{F}_1,\mathcal{F}_2,\ldots\in\mathfrak{F}$ be a family of increasing filters. If each $\mathcal{F}_n$ is $\omega$-diagonalized by $\mathcal{A}_n$, then $\bigcup_{n\in\omega}\mathcal{F}_n$ is $\omega$-diagonalized by $\bigcup_{n\in\omega}\mathcal{A}_n$. Otherwise, there is $n\in\omega$ such that $\mathcal{F}_n=\mathcal{F}_A$, for some finite set $A$. Then all of the filters $\mathcal{F}_k$ for $k\geq n$ are of the form $\mathcal{F}_B$ for $B\subset A$. Since $A$ is finite, this family stabilizes on a filter generated by a certain subset of $A$. Hence $\bigcup_{n\in\omega}\mathcal{F}_n$ is also generated by this subset.
\end{proof}

\begin{prop}
\label{Prop2}
If $\mathcal{F}$ is a P-filter of countable type containing a Fr\'echet filter, then it is diagonalizable.
\end{prop}

\begin{proof}
Let $\mathcal{A}=\left\{A_n:n\in\omega\right\}$ be a family of infinite sets $\omega$-diagonalizing $\mathcal{F}$ and assume that none of the sets $A_n$ diagonalizes $\mathcal{F}$. So, for each $n\in\omega$ there is $B_n\in\mathcal{F}$ such that $A_n\setminus B_n$ is infinite. Since $\mathcal{F}$ is a P-filter, there is $B\in\mathcal{F}$ such that for all $n\in\omega$ we have $B\subset^* B_n$. But then $A_n\setminus B$ is infinite for each $n\in\omega$. This is a contradiction.
\end{proof}

The next example shows that for a filter $\mathcal{F}$ of countable type, ${\rm ct}(\mathcal{F})$ and ${\rm rk}(\mathcal{F})$ can differ.

\begin{exa}
Let $\mathcal{F}=\left\{\omega\right\}\times\mathcal{F}_{Fr}=\left(\left\{\emptyset\right\}\times Fin\right)^*$ (cf. \cite[Ex. 1.2.3]{Farah}) and denote $\mathcal{N}_1^C=\left\{M:\left|C\setminus M\right|<\omega\right\}$. Notice that if $\left(\mathcal{F}_n\right)_{n\in\omega}$ is a sequence of filters of the form $\mathcal{N}_1^{\left\{i\right\}\times\omega}$ and each of them appears in the sequence infinitely many times, then $\mathcal{F}=\lim_{\mathcal{F}_{Fr}}\mathcal{F}_n$. This in turn causes that $\mathcal{F}\in\mathfrak{F}_2$. Consider filters of the form $\lim_{\mathcal{F}_{Fr}}\mathcal{F}_{\left\{x_n\right\}}\in\mathfrak{F}_1$ for some points $x_n\in\omega$ to see that ${\rm ct}(\mathcal{F})=2$. Each such a filter either is not free (if some point repeats infinitely many times in the sequence $\left(x_n\right)_{n\in\omega}$) or is of the form $\mathcal{N}_1^C$ for $C=\bigcup_{i\in\omega}\left\{x_i\right\}$, although $\mathcal{F}$ is a free filter not of the form $\mathcal{N}_1^C$. Finally, notice that there is an infinite set $A=\left\{0\right\}\times\omega$ diagonalizing $\mathcal{F}$ (so its trace on $A$ is a Fr\'echet filter). Any bijection between $\omega$ and $A$ is therefore a quasi-homomorphism from $\mathcal{N}_1$ to $\mathcal{F}$ ensuring that ${\rm rk}(\mathcal{F})=1$.
\end{exa}


\begin{thebibliography}{HD}

\normalsize
\baselineskip=17pt

\bibitem[1]{Debs}   
G. Debs, J. Saint Raymond, Filter descriptive classes of Borel functions, 
Fund. Math. 204 (2009) 189--213.

\bibitem[2]{Marciszewski}   
T. Dobrowolski, W. Marciszewski, Classification of function spaces with the pointwise topology determined by a countable dense set, Fund. Math. 148 (1995) 35--62.

\bibitem[3]{Farah}
Ilijias Farah, Analytic quotients: theory of liftings for quotients over analytic ideals on the integers,
Memoirs of the Amer. Math. Soc. 148 (2000) 1--177

\bibitem[4]{Fremlin}
D. H. Fremlin, Filters of countable type, available at http://www.essex.ac.uk/maths/people/fremlin/preprints.htm (2007), last accessed March 12th 2012.

\bibitem[5]{Grimeisen}   
G. Grimeisen, Ein Approximationsatz f\"ur Baireschen Funktionen, Math. Ann. 146 (1962) 189--194.

\bibitem[6]{Kat}   
M. Katětov, On descriptive classification of functions, Proceedings of the Third Prague Topological Symposium (1972) 235--242.

\bibitem[7]{Kechris}   
A.S. Kechris, Classical descriptive set theory, Grad. Texts in Math., vol. 156, Springer-Verlag, New York, 1995.

\bibitem[8]{Reclaw}   
M. Laczkovich, I. Recław, Ideal limits of sequences of continuous functions, Fund. Math. 203 (2009) 39--46.

\bibitem[9]{Laflamme}   
C. Laflamme, Filter games and combinatorial properties of strategies, Contemp. Math. 192 (1996) 51--67.

\bibitem[10]{Preiss}   
R.D. Mauldin, D. Preiss and H.v. Weizs\"acker, Orthogonal Transition Kernels, Ann. of Probability 11 (1983) 970--988.

\bibitem[11]{Solecki}   
S. Solecki, Filters and sequences, Fund. Math. 163 (2000) 215--228.

\end{thebibliography}
\end{document}